\documentclass[11pt]{amsart}

\usepackage{amscd}
\usepackage[utf8]{inputenc}
\usepackage{lmodern}
\usepackage[a4paper,margin=3cm]{geometry}
\usepackage{amsmath,amsfonts}
\usepackage{enumitem}
\usepackage{xcolor}
\usepackage{hyperref}

\newtheorem{theorem}{Theorem}[section]
\newtheorem{lemma}[theorem]{Lemma}
\newtheorem{proposition}[theorem]{Proposition}

\newtheorem{definition}[theorem]{Definition}
\newtheorem{example}[theorem]{Example}
\newtheorem{Assumption}[theorem]{Assumption}

\theoremstyle{remark}
\newtheorem{remark}[theorem]{Remark}

\newcommand{\dR}{\mathbb{R}}
\newcommand{\cC}{\mathcal{C}}
\newcommand{\cD}{\mathcal{D}}
\newcommand{\cF}{\mathcal{F}}
\newcommand{\dT}{\mathbb{T}}
\newcommand{\dD}{\mathbb{D}}
\newcommand{\dE}{\mathbb{E}}
\newcommand{\dZ}{\mathbb{Z}}
\newcommand{\PAR}[1]{\left(#1\right)}
\newcommand{\abs}[1]{\left | #1\right |}
\newcommand{\norm}[1]{\left \Vert #1\right \Vert}
\newcommand{\set}[1]{\left\{#1\right\}}
\DeclareMathOperator{\osc}{osc}
\let\eps\varepsilon
\newcommand{\equalsDef}{\mathrel{\mathop:}=}
\newcommand{\law}{\text{Law}\,}
\newcommand{\esp}[1]{\mathbb{E}\left[#1\right]}
\newcommand{\ph}{\varphi}
\newcommand{\ind}[1]{\mathbf{1}_{#1}}
\newcommand{\scomma}{\!\!\!,\,}

\newcommand{\MP}{\textbf{MP}}
\newcommand{\sMP}{\textbf{sMP}}
\newcommand{\me}{e}

\numberwithin{equation}{section}
\newcommand{\Hess}{\nabla^2}

\begin{document}

\title{A weak overdamped limit theorem for Langevin processes}
\date{December 1, 2018}

\author{Mathias Rousset}
\address[M. Rousset]{INRIA Rennes - Bretagne Atlantique, $\&$ IRMAR Universit\'e Rennes $1$, France.}
\email[M. Rousset]{mathias.rousset@inria.fr}

\author{Yushun Xu}
\address[Y. Xu and P.~A. Zitt]{Université Paris-Est, Laboratoire d'Analyse et de Mathématiques Appliquées (UMR 8050), UPEM, UPEC, CNRS, F-77454, Marne-la-Vallée, France}
\email[Y. Xu]{Yushun.Xu@u-pem.fr}

\author{Pierre-Andr\'e Zitt}
\email[P.~A. Zitt]{Pierre-Andre.Zitt@u-pem.fr}

\subjclass{35P25,  35Q55}
\keywords{Langevin dynamics, overdamped asymptotics, perturbed test function}

\thanks{This  work  is partially supported  by  the European Research Council under the European Union’s Seventh  Framework  Programme  (FP/2007-2013)/ERC Grant Agreement number 614492}

\begin{abstract}
In this paper, we prove convergence in distribution of Langevin processes in the
overdamped asymptotics. The proof relies on the classical perturbed test
function (or corrector) method, which is used both to show tightness in path space,
and  to identify the extracted limit with a martingale problem. The result
holds assuming the continuity of the gradient of the potential energy, and a
mild control of the initial kinetic energy.
\end{abstract}

\maketitle

\section{Introduction}

 This paper focuses on the overdamped asymptotics of Langevin dynamics.
 The Langevin Stochastic Differential Equation (SDE) describes the dynamics of
 a classical mechanical system perturbed by a stochastic thermostat. The system state at time $t \geq 0$ is encoded by its position $Q_t$ and its
 momentum $P_t$. More formally, the
 equation reads:
 \[
\begin{cases}
  dQ_{t} &= P_t dt,\\
 dP_t &= - \nabla V(Q_t)dt -  P_t dt + \sqrt{2\beta^{-1}}dW_{t},
\end{cases}
\]
where in the above, $Q_t$ takes values in the $d$-dimensional torus
$\mathbb{T}^{d}$, $P_t$ takes values in $\times\dR^{d}$, the function
$V:\mathbb{T}^{d} \to \dR $ is the particles' potential energy, $\beta > 0$ the
inverse temperature, and $t\mapsto W_t \in \dR^d$ is a standard $d$-dimensional
Brownian motion.  The term
$\sqrt{2\beta^{-1}}dW_{t}$ is a fluctuation term bringing energy into the
system, while this energy is dissipated through the friction term
$-P_{t}dt$; the sum of these two terms forming the so-called
thermostat part.  The remaining terms are simply Newton's equation of motion. For
more details on this equation, we refer to~\cite[Section 2.2]{LRS}.

The case we consider here is the so-called overdamped asymptotics, 
where the time scale of the large damping due to friction is much smaller 
than the time scale of the Hamiltonian dynamics, so that the momentum 
becomes a fast variable compared to the slow position variable. 
We introduce a parameter $\eps$ for the ratio of the time scales, and consider 
\begin{align}
  \label{eq:langevinSDE}
\begin{cases}
  dQ^{\eps}_{t}=\frac{1}{\eps}P^{\eps}_{t}dt,\\
 dP^{\eps}_{t}=-\frac{1}{\eps}\nabla V_\eps(Q^{\eps}_{t})dt-\frac{1}{\eps^2}P^{\eps}_{t}dt +\frac{1}{\eps}\sqrt{2\beta^{-1}}dW_{t}.
\end{cases}
\end{align}
Note that we allow the potential $V_\eps\in C^1(\mathbb{T}^d)$ to depend on  $\eps$
and will only suppose that it converges to a limit $V$; see below for a precise 
statement. 
The Markov generator $L_\eps$ associated with~\eqref{eq:langevinSDE} is given by  
\begin{align}
L_\eps f(q,p)
\equalsDef
\frac{1}{\eps^2}\PAR{\frac{1}{\beta} \Delta_p f - p\cdot\nabla_p f }
  +\frac{1}{\eps} \PAR{ p\cdot\nabla_q f - \nabla_q V_\eps \cdot\nabla_p f },
\end{align}
where $f$ denotes any smooth test function of the variables $(q,p) \in \mathbb{T}^d \times \mathbb{R}^d$. \medskip

Overdamped processes are stochastic dynamics on the system position
$(Q_t)_{t\geq 0}$ only. The overdamped Langevin SDE is given by:
\begin{equation}
  \label{eq:overdamped}
    dQ_t = -\nabla V(Q_t)dt+\sqrt{2\beta^{-1}}dB_t,
\end{equation}
where $V: \dT^d \to \dR^d$ is a potential energy, limit of $V_\eps$ when $\eps \to 0$ in some appropriate sense, and  $t\mapsto B_{t} \in \dR^d$ is a standard d-dimensional Wiener process. The Markov generator $L$ associated with~\eqref{eq:overdamped} acts on smooth test functions $f$ of the variable $q$ as follows:
$$Lf(q)\equalsDef-\nabla_q V\cdot\nabla_q f+\frac{1}{\beta}\Delta_qf.$$

Our main result is the proof of the convergence in distribution of the Langevin
position process $\PAR{Q_t^\eps}_{t \geq 0}$ towards its overdamped counterpart
$\PAR{Q_t}_{t \geq 0}$, assuming the uniform convergence of the 
gradient potential as well as
a control of moments of the initial kinetic energy.

\begin{theorem}[Overdamped limit of the Langevin dynamics] 
  \label{thm:overdampedLimit} 
For any $\eps >0$, suppose that  $(Q^\eps_t, P^\eps_t)_{t\geq 0}\in
\mathbb{T}^d\times\dR^d$ is a weak solution to the SDE~\eqref{eq:langevinSDE}. 
Assume that the following conditions hold: 
\begin{enumerate}
  \item 
     $V_\eps$ is $C^1(\mathbb{T}^d)$, and converges to $V$ in the sense that
    $\|\nabla V_\eps -\nabla V\|_{\infty} \xrightarrow[\eps\to 0]{} 0$, 
  \item The following moment bound holds true: 
    \[
      \lim_{\eps \to 0} \eps \mathbb{E}(|P^\eps_0|^3) = 0
    \]
  \item The initial position distribution is converging to some limit: 
    $\law(Q^\eps_0)\xrightarrow[\eps\to 0]{} \law(Q_0)$. 
\end{enumerate}
Then, when $\eps\to 0$, the process $(Q^\eps_t)_{t\geq 0}\in
C(\dR_{+}\to \mathbb{T}^d)$  converges in distribution to the
unique weak solution of the overdamped SDE~\eqref{eq:overdamped}.
\end{theorem}

\begin{remark} In Theorem~\ref{thm:overdampedLimit}, the space of trajectories
  $C(\dR_{+}\mapsto\mathbb{T}^d)$ is endowed with uniform convergence on
  compact sets; making it Polish (metrizable for a separable and complete
  metric).
\end{remark}

The literature on diffusion approximations is very rich; we refer for instance
to Stuart-Pavliotis in~\cite{PS} for a recent pedagogical overview of related
issues. Historically, a possible chain of seminal references is given by
Stratonovich in~\cite{STR}, Khas'minskii in~\cite{KHAS}, Papanicolaou-Varadhan
in~\cite{PV}, as well as Papanicolaou-Kohler in~\cite{PK}; complemented with
the more modern viewpoint of Ethier-Kurtz in~\cite{EK}, Chapter~$12$ "Random
evolutions".

In the present case, the momentum variable is averaged out with the diffusion
approximation, so that the problem may be labeled as ``diffusion approximation
with averaging''. Broadly speaking, the problem can be approached using strong
or weak convergence techniques. For an example of the strong convergence
approach, the results in~\cite{SS} rely on estimating the dynamics of
$Q^\eps_t$ and its limit using a Gronwall argument; this approach requires the
Lipschitz continuity of $\nabla V_\eps$ uniformly in $\eps$. Similar strong convergence results for more advanced models (infinite dimensional, inhomogeneous in space) can be found for instance in~\cite{Cerrai_Freidlin_06,hmvw15}.

On the other hand, weak convergence results rely on the so-called "perturbed" test function or
"corrector" approach, that have been developed since
Panicolaou-Stroock-Varadhan in~\cite{CP}. The case of the overdamped
limit~\eqref{eq:langevinSDE} is not directly covered  by these results. Indeed,
the correctors are not bounded in the present case, due to the fact that the
state space of the momentum variable is not compact.

In a series of papers~\cite{PV1, PV2, PV3}, Pardoux-Veretennikov extend the
classical diffusion approximation with averaging to the non-compact state space
case. In the latter setting however, the slow variable has a dynamics
independent of the fast one, which is not the case in the Langevin
case~\eqref{eq:langevinSDE}. 

We now give a physically motivated example that satisfies our assumptions 
but was not covered by previous works. 
\begin{example}
Let 
\[
V_\eps(q)=V(q)+\alpha_\eps\chi(k_\eps q),
\]
where $\chi\in C^\infty(\mathbb{T}^d)$, and the scaling coefficients $k_\eps \in\mathbb{N}$
and $\alpha_\eps \in \dR$ satisfy 
\[k_\eps \to \infty, \qquad \alpha_\eps k_\eps \to 0.\]

Physically, the potential $\alpha_\eps\chi(k_\eps q)$ may model the interaction
between a particle with unit energy and a periodic crystal of small period
$k^{-1}_\eps$, and small energy range of order~$\alpha_\eps$. When $k_\eps \to
+\infty$ but $\alpha_{\eps}k_\eps=1$ and $\eps$ is kept constant, the effective
action of the periodic crystal on the particle can not be neglected, especially
for grazing velocities co-linear to the principal directions of the crystal.
Indeed, in the latter case, on times of order $1$, the crystal exerts on the
particle a total force also of order $1$, making it deviating from its
trajectory.

Our result shows that the physically necessary condition $\alpha_{\eps}k_\eps
\to 0$ is in fact sufficient for neglecting the crystal effect in the
overdamped regime. Note that if $\alpha_\eps k_\eps^2 \to +\infty$, when
$\eps\to 0$, then 
\[ 
  \|\nabla V_\eps -\nabla V\|_{\infty} \xrightarrow{\eps\to 0} 0,
\]
but still
$$ \|\nabla^2 V_\eps \|_{\infty} \sim  \alpha_\eps k_\eps^2 \|\nabla^2 \chi \|_{\infty}\xrightarrow{\eps\to 0} + \infty,$$
preventing $\nabla V_\eps$ from being Lipschitz uniformly in $\eps$; and hence
forbidding results based on strong convergence.
\end{example}

In order to prove Theorem~\ref{thm:overdampedLimit}, we will establish a more 
general weak convergence result. We consider a sequence (indexed by a small parameter $\eps>0$) of Markov processes of the form $t\mapsto
(Q^\eps_t,P^\eps_t)\in\mathbb{T}^d\times\dR^d$
taking value in the Skorokhod path space
$\mathbb{D}_{\mathbb{T}^d\times\dR^d}$. Our general
convergence result, namely Theorem \ref{th:theorem_gen}, 
gives general conditions under which  $(Q^\eps_t)_{t\geq 0}$ converges in
distribution to the unique solution of a particular martingale problem.
The proof follows the usual pattern: first we  prove tightness for the family
of distributions of $(Q_t^\eps)$,  and then characterize the limit through
martingale problems. For both steps, we use the perturbed test function method.
The key sufficient criteria yielding the results of both steps is given in
Assumption~\ref{assm:perturbed}, which states that to any
smooth $f:\dT^d \to \dR$, we can associate a perturbed test function $f_\eps:\dT^d \times \dR^d
\to \dR$ such that for all $T >0$, 
\[
  \lim_{\eps\to 0} 
    \mathbb{E}\PAR{ \sup_{t\leq T} \abs{ f(Q^\eps_t) - f_\eps(Q^\eps_t,P^\eps_t)}}=0
  \;\text{and}\; 
\lim_{\eps\to 0} \mathbb{E}\PAR{
  \int^T_0 \abs{Lf(Q^\eps_t) - L_\eps f_\eps(Q^\eps_t,P^\eps_t) } dt} =0.
\]

\begin{remark}[On the choice of the state space]
  Theorem \ref{th:theorem_gen} can be useful for c\`ad-l\`ag processes, which
  explains the fact that we work in Skorokhod space.  We have chosen to work in
  $\mathbb{T}^d\times\dR^d$ for notational simplicity, but
  Theorem~\ref{th:theorem_gen} could be extended to more general product spaces
  of the type $E \times F$, where $E$ and $F$ are Polish spaces. If $E$ is
  compact, the extension is straightforward. If $E$ is locally compact, then
  one can work with $E \cup \left\{\infty\right\}$, the one point
  compactification of $E$ at infinity (see~\cite[Chapter~$4$]{EK}). If $E$ is
  not locally compact, then one needs to use Theorem~$9.1$ in
  \cite[Chapter~$3$]{EK} instead of Theorem~\ref{thm:tightnessFromObservables}
  below which is a corollary of the former. In the latter case: (i) the {\it a
    priori} compact containment condition~$(9.1)$ of Theorem~$9.1$ in
  \cite[Chapter~$3$]{EK} has to be proven; and (ii) one has to show the
  tightness of $\big( \law\PAR{f(Q^\eps_t)}_{t \geq 0} \big)_{\eps \geq 0}$ for all
  $f$ in a space of functions dense in $C_b(E)$ for the topology of uniform
  convergence on compacts. Such extensions to infinite dimensional spaces are
  left for future work.
\end{remark}

The paper is organized as follows. Section~\ref{sec:notation} starts with some
notation and preliminaries. In Section~\ref{sec:general_convergence}, we state
and prove the general convergence result Theorem~\ref{th:theorem_gen}.  This general method is then applied
in Section~\ref{sec:langevin} to the overdamped Langevin limit, proving
Theorem~\ref{thm:overdampedLimit}.

\section{Notation and Preliminaries}
\label{sec:notation}
In what follows, we introduce notation and recall some known results.

\subsection{General notation}
Let $(E,d)$ be a Polish space, that is,  a topological space which is metric,
complete and separable. Denote $C(E)$ the Banach space of all continuous
functions and $C_b(E)$ the Banach space of all bounded continuous functions.
We denote by  $\mathcal{P}(E)$  the space of probability measures on
the Borel $\sigma$-field $\mathcal{B}(E)$. The notation $\cF^X_t$ means the natural filtration
of c\`ad-l\`ag processes $(X_t)_{t\geq 0}$, that is
$\cF^X_t=\sigma(X_s, 0\leq s\leq t)$. For any $(s,t)\in
\dR\times\dR$, we denote by $s\wedge t$ the minimum of $s$ and
$t$. 

\subsection{The Skorokhod space}
 A càd-làg function (from the French "continu à droite, limit\'e à  gauche", also
 called RCLL for "right continuous with left limits") is a function
 defined on $\dR_+$ that is everywhere right-continuous and has left
 limits everywhere. The collection of  càd-làg functions on a given
 domain is known as the Skorokhod space.  We denote $\mathbb{D}_E$ the space of
 càd-làg functions with values in a Polish space $E$.  We recall that this path
 space $\mathbb{D}_E$ may be equipped with the Skorokhod topology (see Section~$5$ of~\cite[Chapter~$3$]{EK}): a family of
 trajectories $(q^\eps_s)_{s\geq 0}$ indexed by $\eps$ converges to a
 limit trajectory $(q^0_s)_{s\geq 0}$ if there exists a sequence
 $(\lambda_{\eps})_{\eps\geq 0}$ in  the space of strictly
 increasing continuous bijections of $[0,\infty[$, such that for each $T>0$:
 $\lim_{\eps\to 0}\sup_{t\leq T}|\lambda_{\eps}(t)-t|=0$ and
 $\lim_{\eps\to 0}\sup_{t\leq T}
 d\PAR{ q^{\eps}_{t} , q^0_{\lambda_{\eps}(t)}}=0$.
The following result will be useful in the proof of Theorem~\ref{th:theorem_gen}.
\begin{lemma}
  \label{lem3.5}
Integration with respect to time is continuous with respect to 
the Skorokhod topology: if $(q^\eps_t)_{t\geq 0}$ converges to
$(q^0_t)_{t\geq 0}$ in $\mathbb{D}_E$, and  $\psi:E\to\dR$ is bounded and continuous, 
then for each $T>0$,
\[
  \int^T_0\psi(q^\eps_t)dt \xrightarrow[\eps\to0]{} \int^T_0\psi(q^0_t)dt.
\]
\end{lemma}
\begin{proof} 
  Let us denote by $J_T \equalsDef \left\{ t \in[0,T], q^0_{t^-} \neq q^0_t
  \right\}$ the countable set of jump times in $[0,T]$ of~$q^0$. By definition
  of convergence in the Skorokhod space,
  \[
    \lim_{\eps \to 0} q^\eps_s = q^0_s \qquad \forall s \in [0,T] \setminus J_T. 
  \]
Since $J_T$ has Lebesgue measure $0$ and $\psi$ is continuous and bounded, dominated convergence yields the result.
\end{proof}

\subsection{Martingale problems}
Let us first recall some basics on martingales and stochastic calculus. Let
$\PAR{\Omega,\cF,\textbf{P}, (\cF_t)_{t\geq 0}}$ a filtered probability space.
A c\`ad-l\`ag real-valued process $(X_t)_{t\geq 0}$ is said to be adapted if
$X_t$ is $\cF_t$-measurable for all $t\geq 0$, and is called a $(\cF_t)_{t\geq
  0}$-martingale if $\dE( \abs{X_{t}} | \cF_s) < + \infty$ and $\dE(X_{t} |
\cF_s) = X_{s}$ for any $0 \leq s\leq t$.

We will often need the technical tool of localization by stopping times, 
to deal with the unboundedness of the momentum variable. 
We follow here the presentation of \cite[Chapter 4]{EK}.
\begin{definition}[Local martingale]
  A c\`ad-l\`ag real-valued process $(X_t)_{t\geq 0}$ defined on
  $\PAR{\Omega,\cF,\textbf{P},  (\cF_t)_{t\geq
    0}}$ is called a \emph{local martingale} with respect to
  $(\cF_t)_{t\geq 0}$ if there exists a non-decreasing sequence
  $(\tau_n)_{n\in\mathbb{N}}$ of $(\cF_t)_{t \geq 0}$-stopping times such that
  $\tau_n\to\infty$ $\textbf{P}$-almost surely, and for every $n\in\mathbb{N}$,
  $\big(X_{t\wedge \tau_n}\big)_{t\geq 0}$ is an $(\cF_t)_{t \geq 0}$-martingale.
\end{definition}

Let us now state precisely what it means for a
process to solve a martingale problem. 
\begin{definition}[Martingale problem]
  Let $E$ be a Polish space. Let $L$ be a linear operator mapping a given space
  $\cD \subset C_b(E)$ into bounded measurable functions. Let $\mu$ be a
  probability distribution on~$E$. A càd-làg process $(X_t)_{t\geq 0}$ with
  values in  $E$ solves the \emph{martingale problem} for the generator $L$ on
  the space  $\cD$ with initial measure $\mu$ --- in short, $X$ solves
  $\MP(L,D(L),\mu)$ --- if $\law(X_0) = \mu$ and  if, for any $\varphi\in \cD$, 
\begin{equation}
  \label{SR}
  t\mapsto M_t(\varphi) \equalsDef \varphi(X_t)-\varphi(X_0)-\int^t_0 L\varphi(X_s)ds
\end{equation}
is a martingale with respect to the natural filtration $\PAR{\cF^X_t=\sigma\PAR{X_s, \, 0 \leq s \leq t}}_{t \geq 0}$. 

Moreover, the martingale problem $\MP(L,\cD,\mu)$ is said to be \emph{well-posed} if:
\begin{itemize}
  \item There exists a probability space and a  c\`ad-l\`ag process defined on it 
    that solves the martingale problem (existence); 
  \item whenever two processes solve $\MP(L,\cD,\mu)$, then they have the 
    same distribution on $\mathbb{D}_E$ (uniqueness).
\end{itemize}
\end{definition}

\subsection{Weak solutions of SDEs}
Let $b: \dR^d\mapsto \dR^d$ and $\sigma: \dR^d\mapsto \dR^{d\times n}$ be
locally bounded. Consider a stochastic differential equation in $\dR^d$ of the
form:
\begin{equation}
  \label{eq:weakSDE}
dX_t = b(X_t)dt + \sigma(X_t)dW_t,
\end{equation}
with an initial condition  $\law(X_0)=\mu_0$. Let $L$ be the formal generator 
\begin{align}\label{gene}
L\equalsDef\sum^d_{i=1}b_i\partial_i+\frac{1}{2}\sum^d_{i,j=1}a_{ij}\partial_i\partial_j,
\end{align}
where $a=\sigma\sigma^T$.

\begin{definition}[Weak solution of the SDE]
A continuous process $(X_t)_{t\geq 0}$ is a \emph{weak solution} of 
\eqref{eq:weakSDE} if there exists a filtered probability space
$\PAR{ \Omega,\cF,\textbf{P},  (\cF_t)_{t\geq 0} }$
such that: 

\begin{itemize}
  \item $t\mapsto W_t$ is a $(\cF_t)_{t\geq 0}$-Brownian motion,
    that is, an  $(\cF)_{t\geq 0}$-adapted process 
    such that  $Law(W_{t+h}-W_t|\cF_t)=\mathcal{N}(0,h).$
  \item $X$ is a continuous, $(\cF_t)_{t \geq 0}$-adapted process and satisfies the
stochastic integral equation 
\[
  X_t=X_0+\int^t_0 b(X_s)ds+\int^t_0\sigma(X_s)dW_s \quad a.s.
\]
\end{itemize}
\end{definition}

We now quote two results from~\cite{EK} concerning existence and uniqueness of solutions
to SDEs and martingale problems. 
The first is an existence result, and can be found in \cite[Section 5.3]{EK} (Corollary 3.4 
and Theorem 3.10). 
 
\begin{theorem}
  \label{exist}
Assume that $b$, $\sigma$ are continuous.  If there exists a constant $K$ such that for any $t\geq 0$, $x\in\dR^d$: 
  \begin{align}
    \label{eq:boundOnSigma}
  |\sigma|^2\leq K(1+|x|^2);\\ 
  \label{eq:boundOnB}
  x\cdot b(x)\leq K(1+|x|^2),
  \end{align}
  then there exists a weak solution of the stochastic differential equation
  \eqref{eq:weakSDE} corresponding to $(\sigma,b,\mu)$,  which is also solution of
  the martingale problem $\MP(L,C^\infty_c(\dR^d),\mu)$, $C^\infty_c(\dR^d)$ being the set of smooth functions with compact support. 
\end{theorem}

\begin{remark}  
For the Langevin equation~\eqref{eq:langevinSDE}) we first remark that the latter can be set in $\dR^d \times \dR^d$ using the $\dZ^d$-periodic extension of $V_\eps$. Then
$b(q,p)=\PAR{ \frac{1}{\eps}p, -\frac{1}{\eps}\nabla V_\eps(q)-\frac{1}{\eps^2}p}$ and
$\sigma=(0,\frac{1}{\eps}\sqrt{2\beta^{-1}} \operatorname{Id}_{\dR^d})$ are continuous since $V_\eps\in
C^1(\mathbb{R}^d)$. Moreover, 
\(
|\sigma|^2 = \sigma\sigma^\top=(0,\frac{2}{\beta\eps^2}\operatorname{Id}_{\dR^d}),
\)
and on the other hand
\[
(q,p)\cdot b(q,p)
= \frac{1}{\eps}pq - \frac{1}{\eps}p\nabla V_\eps(q) - \frac{1}{\eps^2}p^2
\leq \frac{1}{2\eps}(1+\|\nabla V_\eps\|_\infty)(1+|p|^2+|q|^2),
\]
which implies the existence of weak solution of~\eqref{eq:langevinSDE} in $\dR^d$. One then   obtains existence of a weak solution in $\dT^d$ of the original~\eqref{eq:langevinSDE} using the canonical continuous mapping $\dR^d \to \dT^d \equalsDef \dR^d/\dZ^d$.
\end{remark}

The next result follows from \cite{EK} (Theorem 1.7 in Section
8.1)  and  \cite{SVR} (Theorem 10.2.2 and the discussion
following their Corollary 10.1.2) .
\begin{theorem}
Assume that the bounds \eqref{eq:boundOnSigma} and \eqref{eq:boundOnB} hold.
Suppose that  $a\equalsDef \sigma\sigma^\top$ is continuous and uniformly
elliptic: 
\[
  \exists C_a > 0, \forall \xi\in\dR^d, \forall x\in\dR^d, 
  \quad 
  \xi^\top a(x)\xi\geq C_a |\xi|^2. 
\]
 Then for any initial condition $\mu$, there is a unique weak solution of the stochastic differential equation
  \eqref{eq:weakSDE}. This solution  is also the unique solution of the martingale problem  $\MP(L, C^{\infty}_c(\dR^d),\mu)$. 
\end{theorem}
\begin{remark}  
For the overdamped Langevin equation~\eqref{eq:overdamped}, we remark again
that the latter can be set in $\dR^d $ using the $\dZ^d$-periodic extension of
$V_\eps$. One then   obtains well-posedness of the martingale problem  $\MP(L,
C^{\infty}_c(\dR^d),\mu)$ in $\dR^d$ since $\nabla V$ is bounded and continuous
by assumption. This solution obviously solves $\MP(L, C^{\infty}(\dT^d),\mu)$
in $\dT^d$. The fact that uniqueness of $\MP(L, C^{\infty}_c(\dR^d),\mu)$
implies uniqueness of $\MP(L, C^{\infty}_c(\dT^d),\mu)$ is technically less
obvious. It can be treated using the localization technique of
Theorem~\ref{th:local_mart} stated in appendix. More precisely, using the
notation of Theorem~\ref{th:local_mart}, one can defines the covering of
$\dR^d$ by the open sets $$U_k\equalsDef\set{(x_1, \ldots, x_d) \in \dR^d |
  \abs{x_i - k_i/8} \leq 1/4 \, \, \forall i=1 \ldots d}$$ where $k \in \dZ^d$
and then remark that by partition of unity for smooth functions, any $\ph \in
C^\infty_c(\dR^d)$ can be written as a finite sum of smooth functions with
compact support in each given $U_k$, $k\in \dZ^d$.
\end{remark}

\subsection{Convergence in distribution}

As we said before, we are interested here in proving convergence in distribution
for processes. Let us briefly recall several  key results that will be used later. \medskip

For completeness, we start by recalling the very classical Prohorov theorem, 
characterizing relative compactness by tightness (see for example Section 2 in\cite[Chapter~$3$]{EK}).
\begin{theorem}[Prohorov theorem]
  \label{prohorov}
  Let $(\mu_{\eps})_{\eps}$ be a family of probability measures on a
  Polish space $E$. Then the following are equivalent:
\begin{enumerate}
  \item $(\mu_{\eps})_{\eps}$ is relatively compact for the topology of convergence in distribution.
  \item $(\mu_{\eps})_{\eps}$ is tight, that is to say, for any $\delta>0$, there is a compact set $K_{\delta}$ such that
    \[
        \inf_{\eps}\mu_{\eps}(K_{\delta})\geq 1-\delta.
      \]

\end{enumerate}
\end{theorem}

Over the years several relative compactness criteria in Skorokhod space have been developed. We will use the following one \cite[Theorem~$8.6$, Chapter~$4$]{EK}.
\begin{theorem}[Kurtz-Aldous tightness criterion]
  \label{thm:KurtzAldous}
  Consider a family of stochastic processes $\PAR{(X^\eps_t)_{t\geq 0}}_{\eps}$ in $\mathbb{D}_{\dR}$.
  Assume that $\big(Law(X^\eps_0)\big)_{\eps}$ is tight. 
  $\forall \delta\in(0,1)$ and $T>0$, there exists a family of nonnegative
  random variable $\Gamma_{\eps, \delta}$, such that: $\forall$ $ 0\leq t\leq
  t+h\leq t+\delta\leq T$
\begin{align}
\mathbb{E}\Big(|X^\eps_{t+h}-X^\eps_t|^2 |\cF^{X^\eps}_t\Big)\leq\mathbb{E}\big(\Gamma_{\eps,\delta}|\cF^{X^\eps}_t\big);
\end{align}
with
\begin{align}\label{eq:kurtz_aldous_bound}
\lim_{\delta\to 0}\sup_{\eps}\mathbb{E}(\Gamma_{\eps,\delta})=0.
\end{align}
  Then the family of distributions $(\law((X^\eps_t)_{t\geq 0}))_\eps$ is tight. 
\end{theorem}
\begin{remark}[On using sequences]
 If $\eps >0$ is a real number and that instead of~\eqref{eq:kurtz_aldous_bound}, one considers the condition
 $\lim_{\delta\to 0}\limsup_{\eps \to 0^+}\mathbb{E}(\Gamma_{\eps,\delta})=0,$ then the conclusion becomes the following: $(\law((X^{\eps_n}_t)_{t\geq 0}))_{\eps_n}$ is tight for any $(\eps_n)_{n \geq 1}$-sequence such that $\eps_n > 0$ and $\lim_{n \to + \infty} \eps_n = 0.$
 This version will be the one used in the present paper.
\end{remark}

If the processes, say $(Q_t^\eps)_{t \geq 0}$, is defined in a general state space $E$, 
it is natural to consider the image processes $(f(Q_t^\eps))_{t \geq 0}$ for 
various observables, or test functions, $f$. The following result
enables us to recover the tightness for the original process from 
the tightness of the observed processes (Corollary 9.3 Chapter 3 in \cite{EK}).

\begin{theorem}[Tightness from observables]
  \label{thm:tightnessFromObservables}
  Let $E$ be a compact Polish space and $\PAR{(Q^\eps_t)_{t\geq
    0}}_{\eps > 0}$ be a family of stochastic processes in
  $\mathbb{D}_E$. Assume that there is an algebra of test functions
  $\mathcal{D}\subset C_{b}(E)$, dense for the uniform convergence,
  such that for any $f\in \mathcal{D}$, 
  $\PAR{(f(Q^{\eps}_{t}))_{t\geq 0} }_{\eps>0}$ is tight in $\mathbb{D}_{\dR}$. Then
  $\PAR{ Law(Q^{\eps}_t)_{t\geq 0} }_{\eps>0}$ is tight in
  $\mathbb{D}_{E}$.
\end{theorem}
\begin{remark}
 Again, the above theorem will be used for families indexed by sequences $(\eps_n)_{n \geq 1}$ such that $\eps_n > 0$ and $\lim_{n \to + \infty} \eps_n = 0.$
\end{remark}

Finally, the following two lemmas will be useful when we 
considering martingale problems. The first one states that the distribution of jumps of
c\`ad-l\`ag processes have atoms in a countable set 
(see Lemma 7.7 Chapter 3 in \cite{EK}).

\begin{lemma}
  \label{lem:fewJumps}
Let $(X_{t})_{t\geq 0}$ be a random process in the Skorokhod path space $\mathbb{D}_{E}$. The set of instants where no jump occurs almost surely:
\begin{align*}
    \mathcal{C}_{Law(X)}\equalsDef\{t\in\dR^{+}|\mathbb{P}(X_{t^{-}}=X_{t})=1\},
\end{align*}
has countable complement in $\dR^{+}$. In particular, it is a dense set.
\end{lemma}

The second one is a very useful way to check whether a process is a
martingale or not (see page 174 in Ethier-Kurtz\cite{EK}). 
\begin{lemma}[Martingale equivalent condition]
  \label{lem:martingaleCondition}
Let $(M_t)_{t\geq 0}$ and $(X_t)_{t\geq 0}$ be two c\`ad-l\`ag proceses and let
$\mathcal{C}$ be an arbitrary dense subset of $\dR_{+}$. Then $(M_t)_{t\geq 0}$
is $\cF^X_t$-martingale if and only if 
\[
  \mathbb{E}\big[(M_{t_{k+1}}-M_{t_k})\varphi_k(X_{t_k})...\varphi_1(X_{t_1})\big]=0,
\]
for any time ladder $t_{1}\leq...\leq t_{k+1}\in \mathcal{C} \subset
\dR_{+}$, $k\geq 1$, and $\varphi_{1},...,\varphi_{k}\in C_{b}(E)$.
\end{lemma}

\section{A general perturbed test function method}
\label{sec:general_convergence}

In this section, we consider a sequence  of stochastic
processes, indexed by a small parameter $\eps>0$,  of the form 
\begin{align*}
t\mapsto (Q^\eps_t,P^\eps_t)\in\mathbb{T}^d\times\dR^d,
\end{align*}
taking value in the Skorokhod path space $\mathbb{D}_{\mathbb{T}^d\times\dR^d}$
associated with the (Polish) product state space $\mathbb{T}^d\times\dR^d$.
Our goal is to 
describe a general framework to prove the convergence of the (slow) variables
$Q$ towards a well-identified dynamics. We use standard tightness arguments and
characterization through martingale problems, emphasizing the  technical role of
perturbed  test functions. 

\subsection{Notation and Assumptions}
  For each $\eps$, we consider a c\`ad-lag process 
  $t\mapsto (Q^\eps_t, P^\eps_t)\in\mathbb{T}^d\times\dR^d$. 
The natural filtration of the full process and the
process $(Q^\eps_t)_{t\geq 0}$ are denoted respectively by 
  $\cF^{Q^\eps\scomma P^\eps}_t \equalsDef \sigma\PAR{(Q^\eps_s,P^\eps_s),0\leq s\leq t }$, and $\cF^{Q^\eps}_t \equalsDef \sigma\PAR{ Q^\eps_s,0\leq s\leq t }$.
We now state  the key assumptions that will imply convergence in distribution
of the process $(Q^\eps_t)_{t\geq 0}$ towards the solution of a martingale
problem. 

\begin{Assumption}[Generator of the process $(Q_t^\eps,P_t^\eps)$ ]
  \label{assm:LepsGenerates}
  There exists a linear operator $L_\eps$ acting on
  $C^\infty(\mathbb{T}^d\times\dR^d)$ which is the extended Markov generator of $(Q_t^\eps,P_t^\eps)_{t \geq 0}$  in 
 the sense that,   for all $ f\in C^\infty(\mathbb{T}^d\times\dR^d)$, $L_\eps f$ is locally bounded and  
\[
  t\mapsto M^\eps_t(f) \equalsDef f(Q^\eps_t, P^\eps_t)-f(Q^\eps_0, P^\eps_0)-\int^t_0 L_\eps f(Q^\eps_s, P^\eps_s)ds
\]
 is a $(\cF^{Q^\eps \scomma P^\eps }_t)_{t\geq 0}$-local martingale. 
 \end{Assumption}

\begin{Assumption}[The limit process]
  \label{assm:well_posed} 
  There exists a linear operator  $L$ mapping $C^\infty(\mathbb{T}^d)$ to $C(\mathbb{T}^d)$ 
  such that the martingale problem $\MP(L,C^\infty(\mathbb{T}^d),\mu)$ 
 is well-posed for any initial condition $\mu$.
\end{Assumption}

\begin{Assumption}[Initial condition]
  \label{assm:initial} 
  The initial condition $\PAR{Law(Q^\eps_0)}_{\eps>0}$
  converge to a limit $\mu_0$, when $\eps\to 0$.
\end{Assumption}

\begin{Assumption}[Existence of perturbed test functions]
  \label{assm:perturbed}
  For all $f\in C^\infty(\mathbb{T}^d)$, there exists a perturbed test function
  $f_{\eps}\in C^\infty(\mathbb{T}^d\times\dR^d)$, such that for all $T$, 
  the rest terms 
  \[
    R^{\eps}_{1,t}(f) \equalsDef \abs{ f(Q^\eps_t)-f_{\eps} (Q^\eps_t, P^\eps_t)}
    \quad \text{and}\quad 
    R^{\eps}_{2,t}(f) \equalsDef \abs{ Lf(Q^\eps_t) - L_{\eps}f_{\eps}(Q^\eps_t, P^\eps_t) }
  \]
  satisfy the following bounds:
\begin{align}
  \label{3.1}
    \lim_{\eps\to 0}\mathbb{E}\PAR{  \sup_{0\leq t\leq T}R^{\eps}_{1,t}(f) } = 0,\\
    \label{3.2}
    \lim_{\eps\to 0}\mathbb{E}\PAR{ \int^{T}_{0}R^{\eps}_{2,t}(f)dt } = 0.
\end{align}
\end{Assumption}

\subsection{The general convergence theorem}

We are now in position to state our main abstract result. 
\begin{theorem}\label{th:theorem_gen}
Under the Assumptions~\ref{assm:LepsGenerates}, \ref{assm:well_posed}, \ref{assm:initial}, and
\ref{assm:perturbed}, the family $\Big(Law(Q^{\eps}_t)_{t\geq 0}\Big)_{\eps>0}$
converges when $\eps \to 0$ to the unique solution of martingale problem
$\MP(L,C^\infty(\mathbb{T}^d),\mu)$.
\end{theorem}
 
The proof follows the classical pattern, in two steps: we first  prove that the processes
$Q^\eps_t$ are relatively compact in $\mathbb{D}_{\dT^d}$; 
then we show that any possible limit must solve the martingale problem 
$\MP(L,C^\infty(\mathbb{T}^d),\mu )$.

\subsubsection{Step one: The proof of tightness.}
We want to prove that for each sequence $(\eps_n)_{n\geq 1}$ satisfying $\lim_n
\eps_n =0$, $\PAR{Law(Q^{\eps_n}_t) }_{n \geq 1}$ is tight.
By Theorem~\ref{thm:tightnessFromObservables}, it is enough to prove the tightness of
$\PAR{\law\PAR{f(Q^{\eps_n}_t)}}_{n \geq 1}$ for all $f\in C^\infty(\mathbb{T}^d)$.
The latter fact will follow from Theorem~\ref{thm:KurtzAldous},  
if we are able to construct, for any function $f\in C^\infty(\mathbb{T}^d)$
and any $\eps,\delta > 0$ and any $T >0$, a random variable $\Gamma_{\eps,\delta}(f)$ such that for all $0\leq t\leq
t+h\leq t+\delta \leq T$, one has

\begin{align}
  \label{eq:continuityEstimate}
  \esp{ \PAR{ f(Q^\eps_{t+h}) - f(Q^\eps_t)}^2 \middle| \cF^{Q^\eps}_t }
  & \leq \esp{ \Gamma_{\eps, \delta}(f) \middle| \cF^{Q^\eps}_t}, \\
  \label{eq:boundOnGamma}
  \text{where}\quad
  \lim_{\delta\to 0}\limsup_{\eps\geq 0}\esp{\Gamma_{\eps, \delta}(f)}
  &= 0.
\end{align}

We claim that the following variant:
\begin{lemma}\label{lem:var_tight}
  For any $g\in\cC^\infty(\dT^d)$, and any $\delta,\eps,T > 0$, there exists a random
  variable $\Gamma'_{\eps,\delta}(g)$ such that for all $0\leq t\leq t+h\leq t+\delta \leq T$, 
\begin{align}
  \label{3.5}
  \abs{
    \esp{ g(Q^\eps_{t+h}) - g(Q^\eps_t) \middle| \cF^{Q^\eps }_t }}
  & \leq \esp{ \Gamma'_{\eps, \delta}(g) \middle| \cF^{Q^\eps}_t}, \\
  \label{3.6}
  \text{where}\qquad
  \lim_{\delta\to 0}\limsup_{\eps\geq 0}\esp{\Gamma'_{\eps, \delta}(g)}
  &= 0.
\end{align}
\end{lemma}
is a sufficient condition. Indeed, the required estimates \eqref{eq:continuityEstimate}, \eqref{eq:boundOnGamma}
will follow easily from the basic decomposition
\[
\PAR{ f(Q^\eps_t)-f(Q^\eps_{t+h})}^2
= \PAR{ f(Q^\eps_{t+h})}^2 - \PAR{ f(Q^\eps_t)}^2 - 2f(Q^\eps_t) \PAR{f(Q^\eps_{t+h})-f(Q^\eps_t)}.
\]
since we get
\begin{align}
  \esp{
  \PAR{ f(Q^\eps_{t+h}) - f(Q^\eps_{t})}^2 \middle| \cF^{Q^\eps }_t
}
\leq \esp{
  \Gamma'_{\eps, \delta}(f^2) \middle| \cF^{Q^\eps}_t
}
  + 2\|f\|_{\infty} \esp{ \Gamma'_{\eps, \delta}(f) \middle| \cF^{Q^\eps}_t }, 
\end{align}
and it is enough to let 
  $\Gamma_{\eps, \delta}(f) 
  = \Gamma'_{\eps, \delta}(f^2) + 2\|f\|_{\infty} \Gamma'_{\eps, \delta}(f)$
 to conclude.\medskip

 Let us now prove the Lemma~\ref{lem:var_tight}. Let $g$ be an arbitrary smooth function, 
 and let $g_\eps$ be the perturbed test function given by Assumption \ref{assm:perturbed}. 
 An elementary rewriting leads to
 \begin{equation}
   \label{eq:decomposition}
   \begin{split}
g(Q^\eps_{t+h}) - g(Q^\eps_{t})
&= \PAR{ g(Q^\eps_{t+h}) - g_\eps(Q^\eps_{t+h}, P^\eps_{t+h})} 
- \PAR{ g(Q^\eps_{t}) - g_\eps(Q^\eps_t, P^\eps_t)} \\
&\quad - \int^{t+h}_{t} \PAR{ L g(Q^\eps_s)-L_{\eps}g_{\eps}(Q^{\eps}_s,P^\eps_s)} ds 
 + \int^{t+h}_{t} L g(Q^\eps_s)ds\\
&\quad - M^\eps_t (g_\eps) + M^\eps_{t+h}(g_\eps),
\end{split}
\end{equation}
where $(M^\eps_t(g_\eps))_{t \geq 0}$ 
is a local $\cF^{Q^\eps \scomma P^\eps}$-martingale by Assumption~\ref{assm:LepsGenerates}. Let $\tau_n$ be an associated localizing sequence of stopping times. Applying \eqref{eq:decomposition} at times $t\wedge\tau_n$ and $(t+h)\wedge\tau_n$,
  we get
\begin{align*}
&g(Q^\eps_{(t+h)\wedge\tau_n}) - g(Q^\eps_{t\wedge\tau_n}) \\
&\quad= g(Q^\eps_{(t+h)\wedge\tau_n})-g_\eps\PAR{Q^\eps_{(t+h)\wedge\tau_n}, P^\eps_{(t+h)\wedge\tau_n}}
 - \PAR{ g(Q^\eps_{t\wedge\tau_n})-g_\eps(Q^\eps_{t\wedge\tau_n}, P^\eps_{t\wedge\tau_n})}
\\
&\qquad 
-\int_{t}^{t+h} \PAR{ L g(Q^\eps_s)-L_{\eps}g_{\eps}(Q^{\eps}_s,P^\eps_s)} \ind {s\leq\tau_n}ds
 +\int_t^{t+h} L g(Q^\eps_s) \ind{s\leq\tau_n} ds \\
&\qquad 
- M^{\eps}_{t\wedge\tau_n}(g_\eps) + M^{\eps}_{(t+h)\wedge\tau_n}(g_\eps). 
\end{align*}
Taking the conditional expectation with respect to $\cF^{Q^\eps}_t$, the martingale terms
cancel out, and we get: 

\[
\begin{split}
 & \abs{\esp{
g(Q^\eps_{(t+h)\wedge\tau_n})-g(Q^\eps_{t\wedge\tau_n})
\middle|\cF^{Q^\eps}_t
}} \\
&\quad \leq 
\abs{\esp{
g(Q^\eps_{(t+h)\wedge\tau_n}) - g_\eps\big(Q^\eps_{(t+h)\wedge\tau_n}, P^\eps_{(t+h)\wedge\tau_n}\big)
\middle|\cF^{Q^\eps}_t
}} \\
&\qquad
+ 
\abs{\esp{
  g(Q^\eps_{t\wedge\tau_n}) - g_\eps(Q^\eps_{t\wedge\tau_n}, P^\eps_{t\wedge\tau_n})
  \middle|\cF^{Q^\eps}_t
}} \\
&\qquad  + \int^{t+h}_{t} 
\abs{\esp{
  L g(Q^\eps_s)-L_{\eps}g_{\eps}(Q^{\eps}_s, P^\eps_s)
  |\cF^{Q^\eps}_t
}} ds
+ h\sup_{q\in\mathbb{T}^d}\abs{Lg(q)} \\
&\quad \leq 
\esp{
    R_{1,(t+h)\wedge\tau_n}^\eps +  R^\eps_{1,t\wedge\tau_n}
\middle|\cF^{Q^\eps}_t
}
 + \int^{t+h}_{t} 
\esp{
  R_{2,s}^\eps
  |\cF^{Q^\eps}_t
} ds
+ \delta \sup_{q\in\mathbb{T}^d}\abs{Lg(q)} \\
&\quad \leq 
2\esp{
  \sup_{s \in [0,T]} R^\eps_{1,s}
\middle|\cF^{Q^\eps}_t
}
 + \int_{0}^{T} 
\esp{
  R_{2,s}^\eps
  |\cF^{Q^\eps}_t
}
  ds
+ \delta \sup_{q\in\mathbb{T}^d}\abs{Lg(q)}.
\end{split}
\]
The right hand side does not depend on $n$ any longer. On the left hand side, 
we apply dominated convergence for $n\to \infty$ to get

\[
  \abs{\esp{
g(Q^\eps_{(t+h)})-g(Q^\eps_{t})
\middle|\cF^{Q^\eps}_t
}} \leq \esp{ \Gamma'_{\eps,\delta}(g)\middle|\cF^{Q^\eps}_t} 
\]
for $\Gamma'_{\eps,\delta}(g) = 2\sup_{[0,T]}R_{1,t}^\eps + \int_0^T R_{2,s}^\eps ds + \delta \|Lg\|_\infty$. 
The controls on the rest terms given by Assumption~\ref{assm:perturbed}, 
and the continuity of $Lg$ (Assumption~\ref{assm:well_posed})
ensure that 
\[ \lim_{\delta\to 0} \limsup_{\eps\to 0} \Gamma'_{\eps,\delta}(g) = 0,\]
and the proof of tightness is concluded.

\subsubsection{Step two: identification of the limit}
In this step, we suppose that a sequence $Q^n_t = Q^{\eps_n}_t$ 
converges in distribution to a limit $Q^0_t$, and we prove
that necessarily, $Q^0$ solves the martingale problem 
for the generator $L$. 

 Let $f\in C^\infty(\mathbb{T}^d)$, we have to check that
\begin{align}
  M_t(Q^0_t) \equalsDef f(Q^0_t)-f(Q^0_0)-\int^t_0 Lf(Q^0_s)ds
\end{align}
is a martingale with respect to $\cF^{Q^0}_t=\sigma(Q^0_s, 0\leq
s\leq t)$. Consider a time sequence $0\leq t_{1}\leq\cdots\leq
t_{p}\leq t_{p+1}$ for $p\geq 1$, taken in the continuity set $\mathcal{C}_{\law(Q)}$
given by Lemma~\ref{lem:fewJumps}. Recall that $\mathcal{C}_{\law(Q)}$ is dense in $\dR$. Let
$\varphi_{1},...,\varphi_{p}\in C_{b}(\mathbb{T}^d)$ be  $p$ test functions. 
By Lemma \ref{lem:martingaleCondition}, it is enough to prove that
\[
  I_0 \equalsDef
  \esp{\PAR{
    f(Q^{0}_{t_{p+1}})-f(Q^{0}_{t_{p}}) 
    - \int^{t_{p+1}}_{t_p}Lf(Q^{0}_{s})ds} 
    \varphi_{1}(Q^{0}_{t_{1}}) \cdots \varphi_{p}(Q^{0}_{t_{p}})
  }
= 0. 
\]

Let $I_\eps$ be the corresponding quantity for $\eps>0$, that is, 
\[
    I_{\eps} \equalsDef 
    \esp{\PAR{ 
	f(Q^{\eps}_{t_{p+1}}) - f(Q^{\eps}_{t_{p}}) 
	- \int^{t_{p+1}}_{t_{p}}Lf(Q^{\eps}_{s})ds
      } \varphi_{1}(Q^{\eps}_{t_{1}}) \cdots \varphi_{p}(Q^{\eps}_{t_{p}})
    }.
      \]
Let us first show that $I_\eps$ converges to $0$. We first condition on $\cF_{t_p}^{Q^\eps}$ to get:
\begin{align*}
  \abs{I_\eps}
  &\leq 
    \esp{ \esp{ \left| f(Q^\eps_{t_{p+1}}) 
      - f(Q^\eps_{t_p})
	- \int^{t_{p+1}}_{t_{p}} 
      Lf(Q^\eps_s) ds 
      \right|
      \,\,\, \middle | \cF^{Q^\eps}_{t_p} }
    \abs{\varphi_{1}(Q^{\eps}_{t_{1}})} \cdots \abs{ \varphi_{p}(Q^{\eps}_{t_{p}})}} \\
  &\leq 
    \esp{ \esp{ \left| f(Q^\eps_{t_{p+1}}) 
      - f(Q^\eps_{t_p})
	- \int^{t_{p+1}}_{t_{p}} 
      Lf(Q^\eps_s) ds 
      \right|
      \,\,\, \middle | \cF^{Q^\eps}_{t_p} }}
    \|\varphi_1\|_\infty \cdots \|\varphi_p\|_\infty.  
\end{align*}
Using again the perturbed test function $f_\eps$ and the decomposition
\eqref{eq:decomposition}, we get by the same localization argument as in
Step~1 that 
\[
    |I_{\eps}|
    \leq
    \esp{ R^{\eps}_{1,t_{p+1}}(f)+R^{\eps}_{1,t_{p}}(f)
    +\int^{t_{p+1}}_{t_{p}}R^{\eps}_{2,s}(f)} 
    \|\varphi_{1}\|_{\infty}...\|\varphi_{p}\|_{\infty}. 
  \]
The estimates on the rest term from Assumption~\ref{assm:perturbed} 
then imply that $I_\eps\to 0$. \medskip

Let us now prove that $I_\eps$ converges to $I_0$. Let $\Phi:\dD_{\dT^d}\to\dR$ be the functional
\[
\Phi: (q_t)_{t\geq 0}
	   \mapsto 
	   \PAR{ f(q_{t_{p+1}}) - f(q_{t_p})
	     - \int^{t_{p+1}}_{t_p} Lf(q_s)ds} \varphi_{1}(q_{t_1}) \cdots \varphi_{p}(q_{t_p})
 \]
so that $I_\eps = \esp{ \Phi((Q^\eps_t)_{t\geq 0})}$ and $I_0 =
\esp{\Phi((Q^0_t)_{t\geq 0})}$. Let us first check that, if
$q^0 \in \dD_{\dT^d}$ satisfies 
$q^0_{t_k^-}=q^0_{t_k}$ for each $1 \leq k \leq p+1$, 
then the functional $\Phi$ is continuous at the trajectory $q^0$. 
Indeed,  since $Lf$ is continuous and bounded by
Assumption~\ref{assm:well_posed},  Lemma~\ref{lem3.5} shows that the map
$(q_t)_{t\geq 0} \mapsto \int^{t_{p+1}}_{t_p}Lf(q_s)ds$ 
is continuous with respect to Skorokhod topology; 
moreover, by assumption, $q^0$ is continuous at the time $t_k$ for each $1 \leq k \leq p+1$, 
so  the map
 \(
   (q_t)_{t\geq 0} \mapsto \varphi_k (q_{t_k})
 \)
is continuous at $q^0 \in \dD_{\dT^d}$.

Let now $(\eps_n)_{n \geq 1}$ be any sequence such that $\eps_n \to 0$ and
$(Q^{\eps_n}_t)_{t \geq 0}$ converges in distribution to $(Q^{0}_t)_{t \geq
  0}$. The Skorokhod representation theorem (Theorem~$1.8$
in~\cite[Chapter~$3$]{EK}) ensures that one can construct a probability space
where the distribution of $(Q^{\eps_n}_t)_{t\geq 0}$ for each $n$ is unchanged
but for which $\lim_{n \to +\infty} Q^{\eps_n} = Q^{0}$ almost surely in
$\dD_{\dT^d}$. Since $t_k \in \mathcal{C}_{\law(Q^0)}$ for each $k=1 \ldots
p+1$, $\Psi$ is almost surely continuous at $Q^0$ and we can apply the
dominated convergence theorem to obtain $\lim_{n \to +\infty} I_{\eps_n} =
I_0$. Since the choice of the vanishing sequence $(\eps_n)_{n \geq 1}$ is
arbitrary, we conclude that $\lim_{\eps \to 0} I_{\eps} = I_0$. The limit
process thus solves the martingale problem $\MP(L,C^\infty(\mathbb{T}^d),\mu
)$.

\subsubsection{Conclusion.}
For each sequence $(\eps_n)_{n\geq 1}$ satisfying $\lim_n \eps_n =0$, we have
proven that $\PAR{Law(Q^{\eps_n}_t) }_{n \geq 1}$ is tight and that any
converging subsequence is solution to the martingale problem
$\MP(L,C^\infty(\mathbb{T}^d),\mu )$. By uniqueness of the latter according to
Assumption~\ref{assm:well_posed}, this identifies the limit, showing that
$\PAR{Law(Q^{\eps_n}_t) }_{n \geq 1}$ converges to the solution of
$\MP(L,C^\infty(\mathbb{T}^d),\mu )$. Since the sequence $(\eps_n)_{n\geq 1}$
is arbitrary and convergence in distribution is metrizable,
$\PAR{Law(Q^{\eps}_t) }_{\eps > 0 }$ also converges to the solution of
$\MP(L,C^\infty(\mathbb{T}^d),\mu )$, proving Theorem~\ref{th:theorem_gen}.

\section{Overdamped limit of the Langevin dynamics}
\label{sec:langevin}

In the section, we will use the perturbed test function method presented in
last section to prove Theorem \ref{thm:overdampedLimit}. We will first state
the key estimates on $(\abs{P^\eps_t})_{t \geq 0}$.  
These estimates are then used to check the assumptions of our
general Theorem~\ref{th:theorem_gen} in the specific case
of Langevin processes. In a last section  we will detail the 
proof of the key estimates. 

\subsection{Some moments estimates for Langevin processes}
We start by giving a few facts about the solution to the Langevin
SDE~\eqref{eq:langevinSDE}. We first check that the operator $L_\eps$ acting on
$\cC^\infty(\dT^d, \dR^d)$ by
\[
  L_\eps f(q,p)
\equalsDef
\frac{1}{\eps^2}\PAR{\frac{1}{\beta} \Delta_p f - p\cdot\nabla_p f }
  +\frac{1}{\eps} \PAR{ p\cdot\nabla_q f - \nabla_q V_\eps \cdot\nabla_p f }
\]
is the generator the process, in the sense that Assumption~\ref{assm:LepsGenerates} holds. 
 \begin{proposition}\label{prop4.1}
  If $(Q^\eps_t,P^\eps_t)_{t \geq 0}$ is a weak solution of the Langevin
  SDE~\eqref{eq:langevinSDE}, then for any smooth function $f:\dT^d \times
  \dR^d \to \dR^d$, the process
\begin{align*}
    t\mapsto M^{\eps}_{t}(f) 
    = f(Q^\eps_t,P^\eps_t) - f (Q^\eps_0,P^\eps_0) 
    - \int^t_0 L_\eps f (Q^\eps_s,P^\eps_s)ds,
\end{align*}
is a $(\cF^{Q^\eps,P^\eps}_{t})_{t \geq 0}$-local martingale.
\end{proposition}
\begin{proof}
This is a very classical result. By  Itô calculus we write
\[
df_\eps(Q^\eps_t,P^\eps_t)
= L_\eps f_\eps(Q^\eps_t,P^\eps_t)dt+\frac{1}{\eps}\sqrt{2\beta^{-1}}\nabla_p f_\eps(Q^\eps_t,P^\eps_t)dW_t.
\]
Defining the sequence of $\PAR{\cF^{Q^\eps,P^\eps}_t}_{t \geq 0}$-stopping time 
\begin{align}\label{w4.6}
\tau_n=\inf\{t\geq 0, |P^\eps_t|\geq n\},
\end{align}
which converge almost surely to infinity, we obtain that 
\[
M^{\eps, n}_t(f_\eps) \equalsDef \frac{1}{\eps}\sqrt{2\beta^{-1}}\int^t_0\nabla_p f_\eps(Q^\eps_s,P^\eps_s)1_{s\leq \tau_n}dW_s
\]
is a $\PAR{\cF^{Q^\eps,P^\eps}_t}_{t \geq 0}$-martingale for any $n\geq 0$, which is the definition of a local martingale.
\end{proof}

We now state several bounds on the momentum variable $P^\eps_t$, 
which are the key technical estimates needed later to control 
the rest terms appearing in the perturbed test function method. 
For any continuous $V:\dT^d\to \dR$ we denote by $\osc(V)$ the oscillation defined by 
\[ \osc(V) = \max V - \min V.\]

\begin{lemma}[Propagation of moments]
  \label{lem:finiteMoments}
For any $\gamma \geq 1$, any $M>0$ and any $\beta>0$,  there is a numerical constant
$C(\gamma, M, \beta)$ such that for any $\eps > 0$, if $\osc(V_\eps)  \leq M$, 
then 
\begin{equation}\label{eq:finiteMoments}
  \sup_{t \geq 0} \esp{ |P^\eps_t|^{2\gamma}}
  \leq C(\gamma, M, \beta) \PAR{ \esp{ |P^\eps_0|^{2\gamma}} + 1}. 
\end{equation}
\end{lemma}

\begin{lemma}[Moment of suprema]
  \label{lem:boundsOnPt}
  For any $M>0$, any $\beta>0$ and any $T>0$, there is a numerical 
  constant $C(M,\beta,T)$ such that for any $\eps \in (0,1)$, if $\osc(V_\eps)\leq M$, 
  then 
\begin{align}
   \label{eq:espSupPt}
   &\esp{ \sup_{0\leq t\leq T}|P^\eps_t|^2}  
   \leq \esp{|P^\eps_0|^2} + \frac1\eps C(M,\beta,T) \PAR{\esp{|P^\eps_0|^2}+1}^{1/2}.
\end{align}
In particular, if $\lim_{\eps \to 0} \eps^{2} \esp{ |P^\eps_0|^2} = 0$, then  
\[
\lim_{\eps \to 0} \eps^{2} \esp{ \sup_{0\leq t\leq T}|P^\eps_t|^2} = 0.
\]
\end{lemma}
The proofs of these estimates use classical techniques of stochastic
calculus and  are postponed to Section~\ref{sec:proofOfPBounds}. 

\subsection{The perturbed test functions in the Langevin case}
In this section we apply the general method described 
in Section~\ref{sec:general_convergence} to the specific
Langevin case, in order to prove Theorem~\ref{thm:overdampedLimit}. \medskip

We will use the following standard notation for multidimensional derivatives:
$$ \nabla^k f (p_1, \ldots, p_k ) \equalsDef \sum_{i_1, \ldots , i_k =1}^d \partial_{i_1} \ldots \partial_{i_k} f \, \times p_1^{i_1} \, \times \ldots  \, \times p_k^{i_k} $$
where in the above $ p_1, \ldots, p_k \in \dR^d$. Note that as usual $\Delta f = \operatorname{Tr} \PAR{\nabla^2 f}$.

We first construct explicitly, for any $f\in C^\infty(\mathbb{T}^d)$, a
perturbed test function $f_\eps\in C^{\infty}(\mathbb{T}^d\times\dR^d)$. Let
us look for $f_\eps$ in the following form (see~\cite{CP})
\begin{align}\label{4.9}
    f_{\eps}(q, p)=f(q)+\eps g_{1}(q,p)+\eps^2 g_{2}(q,p).
\end{align}
Applying the generator $L_\eps$, using the fact that 
$f$ does not depend on $p$, and grouping terms with respect to 
powers of $\eps$, we get
\begin{align}
  L_\eps f_\eps(q,p)
  &= \frac{1}{\eps} p\cdot\nabla_{q}[f(q)+\eps g_{1}(q,p)+\eps^2 g_{2}(q,p)]
    -\frac{1}{\eps}\nabla_q V(q) \cdot \nabla_p[ \eps g_1(q,p)+\eps^2 g_2(q,p)] \notag\\
  &\quad - \frac{1}{\eps^2} p \cdot\nabla_p[ \eps g_{1}(q,p)+\eps^2 g_{2}(q,p)]
         + \frac{1}{\eps^2\beta } \Delta_p[ \eps g_1(q,p)+\eps^2 g_2(q,p)] \notag\\
  &= \frac{1}{\eps} \PAR{
    p\cdot\nabla_q f - p\cdot\nabla_p g_1
    + \frac{1}{\beta}\Delta_{p}g_{1}} \notag\\
  &\quad + \PAR{ p\cdot\nabla_q g_1
    - \nabla_q V_\eps\cdot\nabla_p g_1
    - p\cdot\nabla_p g_2 + \frac{1}{\beta}\Delta_p g_2 } \notag\\
  &\quad + \eps\PAR{ 
    p\cdot\nabla_q g_2 - \nabla_p g_2 \cdot\nabla_q V_\eps \label{eq:pert_eps}
  }.
\end{align}

In order for $L_\eps f_\eps$ to converge to $Lf$, the $\eps^{-1}$-order terms should vanish, and the $\eps^{0}$-order terms should converge at least formally to $L(f)$. As a consequence $g_1$ and $g_2$ 
should solve the following equations: 
\begin{align}
  \label{eq:g1}
  0 &= p\cdot\nabla_q f  - p\cdot\nabla_p g_1 + \frac{1}{\beta}\Delta_p g_1 , \\
  Lf(q) &= p\cdot\nabla_q g_1 - \nabla_q V\cdot\nabla_p g_1 
  - p\cdot\nabla_p g_2 + \frac{1}{\beta}\Delta_p g_2. \label{eq:g2}.
\end{align}
The function $ g_1(q,p) = p\cdot\nabla_{q}f(q)$ clearly solves \eqref{eq:g1}. 
With this choice, \eqref{eq:g2} becomes
$$
Lf(q) = \Hess_q f (p,p) - \nabla_q V \cdot \nabla_q f - p\cdot\nabla_p g_2 + \frac{1}{\beta}\Delta_p g_2. 
$$
Since $Lf(q) = \frac{1}{\beta}\Delta_q f - \nabla_q V \cdot \nabla_q f$, 
it is easy to check that 
\(
  g_{2}(q,p) = \frac{1}{2} \Hess_{q}f (p,p)
\)
solves the equation. 

Therefore, in view of Eq.~\eqref{4.9},
we defined the perturbed test function by :
\begin{equation}
    f_\eps(q, p) = f(q) + \eps p\cdot\nabla_q f
      + \frac{1}{2} \eps^2 \Hess_q f(p,p). 
\end{equation}
With this choice, we get using previous calculations and the last line of~\eqref{eq:pert_eps}
\begin{align}
& L_\eps f_\eps(q,p) - Lf(q) 
 \notag \\
& \qquad   =\PAR{ \nabla_q V - \nabla_q V_\eps}\cdot\nabla_q f + \eps\PAR{ 
    p\cdot\nabla_q g_2 - \nabla_p g_2 \cdot\nabla_q V_\eps } \notag \\
& \qquad =
  \PAR{ \nabla_q V - \nabla_q V_\eps}\cdot\nabla_q f 
 + \frac{1}{2} \eps\PAR{ \nabla^3_q f(q)(p,p,p) 
  -  \Hess_q f \PAR{ p, \nabla_q V_\eps }}. \label{eq:LepsMinusL}
\end{align}

We now need to show that  Assumption~\ref{assm:perturbed} holds for this
choice of a perturbed test function, that is, we want to show that  the
differences $f_\eps - f$ and $L_\eps f_\eps - Lf$ are small in the following appropriate
sense. Recalling the notation
\begin{align*}
    R^\eps_{1,t}(f) &= \abs{ f(Q^\eps_t) - f_\eps(Q^\eps_t, P^\eps_t) },
    &
    R^\eps_{2,t}(f) &= \abs{ Lf(Q^\eps_t) - L_\eps f_\eps(Q^\eps_t, P^\eps_t) },
\end{align*}
we need to prove that
\begin{align}
  \label{eq:control 1}
    \lim_{\eps\to 0}\mathbb{E}\Big(\sup_{0\leq t\leq T}R^{\eps}_{1,t}(f)\Big)=0,\\
    \label{eq:control 2}
    \lim_{\eps\to 0}\mathbb{E}\Big(\int^{T}_{0}R^{\eps}_{2,t}(f)dt\Big)=0.
\end{align}

Since $f\in C^\infty(\mathbb{T}^d)$, there exists 
a $C_f = \max \PAR{\norm{\nabla f}_\infty,\norm{\nabla^2 f}_\infty}$ such that for all $(q,p)$ and all $\delta \in (0,1/2)$
\begin{align*}
  \abs{ f_\eps(q, p) - f(q)} 
  &=   \eps \abs{ p\cdot\nabla_q f(q)} 
      + \frac{1}{2} \eps^2 \abs{ \Hess_q f(q)\cdot(p,p)} \\
  &\leq C_f (\eps \abs{p} + \eps^2 \abs{p}^2) \\
  &\leq \delta C_f + \frac1\delta C_f \eps^2 \abs{p}^2,
\end{align*}
where we have used that for any $\delta > 0$,  $\eps \abs{p} \leq \frac12 \delta + \frac12 \eps^2 \abs{p}^2/\delta$.
Therefore

\[
  \esp{\sup_{t\in[0,T]} R^\eps_{1,t}(f) }
  \leq \delta C_f +  \frac1\delta C_f \eps^2 
    \esp{\sup_{t\in[0,T]} \abs{P^\eps_t}^2 }. 
\] 
By assumption, $\lim_{\eps \to 0} \eps \esp{\abs{P^\eps_0}^3} = 0$, so
$\eps^2 \esp{\abs{P^\eps_0}^2} \leq  \eps^{4/3} ( \eps \esp{\abs{P^\eps_0}^3})^{2/3}$
also goes to zero by Jensen's inequality. 
By the key Lemma~\ref{lem:boundsOnPt} this entails
that the last term  in the previous display disappears 
in the limit and we get 
\[
  \limsup_{\eps \to 0} \esp{\sup_{t\in[0,T]} R^\eps_{1,t}(f) }
  \leq \delta C_f,
\]
which proves \eqref{eq:control 1} since $\delta$ is arbitrary. \medskip

We now turn to the proof of \eqref{eq:control 2}, that is, we want to compare
$L_\eps f_\eps$ and $Lf$. By the expression~\eqref{eq:LepsMinusL}, 
we have for some constant $C_{f} = \max \PAR{\norm{\nabla f}_\infty,\norm{\nabla^2 f}_\infty,\norm{\nabla^3 f}_\infty}$

\begin{equation*}
  \abs{ L_\eps f_\eps(q,p) -  Lf(q)}
    \leq C_f \| \nabla_q V - \nabla_q V_\eps\|_\infty
          + C_f \eps \PAR{\abs{p}^3 + \| \nabla_q V_\eps\|_\infty \abs{p}}.
\end{equation*}

We get rid of the product term with Young's inequality $ab \leq a^3/3 + \frac{2}{3} b^{3/2}\leq a^3 + b^{3/2}$
and get

\[
   \esp{ R^\eps_{2,t}}
  \leq C_f \|\nabla_q V - \nabla_q V_\eps\|_\infty
     + \eps C_f \esp{2 \abs{P^\eps_t}^3 + \| \nabla V_\eps\|_\infty^{3/2}}. 
   \]
  We integrate in $t$ to obtain
\begin{align*}
  \int_0^T  \esp{ R^\eps_{2,t}} dt
  &\leq C_f \|\nabla_q V - \nabla_q V_\eps\|_\infty T
     + \eps C_f T \PAR{\sup_{t\in [0,T]} \esp{ 2 \abs{P^\eps_t}^3} + \|\nabla V_\eps\|_\infty^{3/2}}.
   \end{align*}
   By assumption, $\lim_{\eps \to 0} \eps \esp{\abs{P^\eps_0}^3}=0$,
   and by the uniform convergence of $\nabla V_\eps$ to $\nabla V$ 
   we can find a uniform bound $M$ such that $\osc(V_\eps) \leq M$ for all 
   $\eps$, so we may apply Lemma~\ref{lem:finiteMoments} with $\gamma=3/2$ 
   and get
   \[ \lim_{\eps \to 0} \eps \sup_{t\in [0,T]}\esp{\abs{P^\eps_t}^3}=0, \]
   for any $T \geq 0$. Together with the convergence of $\nabla V_\eps$ to $\nabla V$
   this yields
   \[ 
     \lim_{\eps \to 0} 
  \int_0^T  \esp{ R^\eps_{2,t}} dt
  = 0.
\]
from which~\eqref{eq:control 2} follows.

\subsection{Proofs of the moment bounds}
  \label{sec:proofOfPBounds}

  We now come back to the proofs of the moment bounds (Lemmas~\ref{lem:finiteMoments}
  and \ref{lem:boundsOnPt}). It will prove useful to work with the 
Hamiltonian of the system rather than directly with $P^\eps_t$. For 
convenience's sake we assume without loss of generality 
that $0\leq V_\eps(q) \leq \osc(V_\eps)$. 
\begin{definition}[Hamiltonian]
  We denote by $H^\eps$ the Hamiltonian of the system:
\[
  H^\eps(q,p) = \frac{1}{2} \abs{p}^2 + V_\eps(q). 
\]
We will also write 
\( H^\eps_t \equalsDef H^\eps(Q^\eps_t, P^\eps_t).\)
\end{definition}
By Itô's formula, 
\begin{align}
  \notag
  dH^\eps_t &= P_t^\eps dP_t^\eps + \nabla_q V_\eps(Q_t^\eps) dQ_t^\eps
  + \frac{1}{2} \sum_{i,j=1}^d d\langle (P^\eps)^i,(P^\eps)^j \rangle_t \\
  \label{eq:itoForH0}
  &= \PAR{ - \frac{1}{\eps^{2}} \abs{P_t^\eps}^2 + \frac{1}{\eps^2\beta}} dt 
      + \frac{1}{\eps}\sqrt{2\beta^{-1}} P_t^\eps dW_t \\
  \label{eq:itoForH}
  &= \PAR{ - \frac{2}{\eps^2} H^\eps_t + \frac{2}{\eps^2} V_\eps(Q^\eps_t)
      + \frac{1}{\eps^2\beta} } dt 
      + \frac{\sqrt{2\beta^{-1}}}{\eps} P^\eps_t dW_t.
\end{align}
Again, by It\^o's formula, we thus get for any smooth function $(t,h) \mapsto \phi(t,h)$
\begin{equation}
  \label{eq:itoForPhi(H)}
  d\phi(t,H^\eps_t) 
  = \partial_t \phi(t,H^\eps_t)dt + \partial_h\phi (t,H^\eps_t) dH^\eps_t 
  + \frac{1}{\eps^2\beta} \partial^2_{h} \phi (t,H^\eps_t) \abs{P^\eps_t}^2 dt.
\end{equation}

\begin{proof}[Proof of Lemma \ref{lem:finiteMoments}]
  Let $\gamma\geq  1$. We apply~\eqref{eq:itoForPhi(H)} to $\phi(t,x)=
  \me^{\alpha t}\, h^\gamma$ and plug in~\eqref{eq:itoForH}
  to get:
  \begin{align*} 
    d(\me^{\alpha t} (H_t^\eps)^\gamma)
    &= 
    \gamma (H_t^\eps)^{\gamma -1}  \PAR{ \frac{\alpha}{\gamma} H_t^\eps
      - \frac{2}{\eps^2} H^\eps_t
      + \frac{2}{\eps^2} V_\eps(Q^\eps_t) 
      + \frac{1}{\eps^2\beta} } e^{\alpha t} dt  \\
    &\quad +  \frac{\sqrt{2\beta^{-1}}}{\eps}
    \gamma  (H^\eps_t)^{\gamma -1} P^\eps_t e^{\alpha t}  d W_t
    + \frac{\gamma(\gamma - 1)}{\eps^2\beta}   (H^\eps_t)^{\gamma -2} \abs{P^\eps_t}^2 e^{\alpha t} dt. 
   \end{align*}
   The choice 
   \[ 
     \alpha =2\gamma / \eps^2
   \]
   cancels the higher order term in the first bracket. 
   We integrate in time, multiply by $e^{-\alpha t}$ and 
   regroup the finite variation terms to get: 
  \begin{align*} 
    (H_t^\eps)^\gamma
    &= (H_0^\eps)^\gamma  
      + \int_0^t
      \PAR{  \gamma (H_s^\eps)^{\gamma -1} 
      \PAR{ \frac{2}{\eps^2} V_\eps(Q^\eps_s)  + \frac{1}{\eps^2\beta} } 
      + \frac{\gamma (\gamma -1)}{\eps^2\beta}   (H^\eps_s)^{\gamma -2} \abs{P^\eps_s}^2
    }
      \me^{-\alpha(t-s)}ds  \\
    &\quad 
    +  \frac{\sqrt{2\beta^{-1}}}{\eps}
      \int_0^t \gamma (H^\eps_s)^{\gamma -1} P^\eps_s \me^{-\alpha(t-s)} d W_s. 
   \end{align*}
   Since $(1/2)\abs{P^\eps_s}^2 \leq H^\eps_s \leq (1/2) \abs{P^\eps_s}^2 + \osc(V_\eps)$, 
\begin{equation}
  \label{eq:boundHk}
  \begin{aligned} 
    (H_t^\eps)^\gamma 
    &\leq  (H_0^\eps)^\gamma  
    + \frac{2 \gamma }{\eps^2} \PAR{\osc(V_\eps) +  \frac{\gamma}{\beta}}
    \int_0^t  (H_s^\eps)^{\gamma -1}  \me^{-\alpha(t-s)} ds   \\
    &\quad   +  \frac{\sqrt{2\beta^{-1}}}{\eps} \int_0^t (H_s^\eps)^{\gamma -1} P^\eps_s \me^{-\alpha(t-s)} dW_s. 
   \end{aligned}
 \end{equation}

To deal with the unboundedness of the momentum $P$, we define the following
stopping times: 
\begin{equation}
  \label{eq:defTauN}
  \tau_n\equalsDef \inf\{ t: \abs{P_t} = n\}. 
\end{equation}
When $s\leq \tau_n$, we have $|P^\eps_s|\leq n$ and $H^\eps_s\leq
(\osc(V_\eps) +\frac{n^2}{2})$. 
This entails that 
\(
t\mapsto \int_0^{ t\wedge \tau_n } (H^\eps_s)^{\gamma -1} P^\eps_s dW_s
\)
is martingale. 
Writing \eqref{eq:boundHk} at time $t\wedge\tau_n$ and taking expectations, 
the martingale part disappears; recalling that  $\alpha =2\gamma / \eps^2$
we get
\begin{align*}
    \esp{ (H_{t\wedge \tau_n} ^\eps)^\gamma }
    &\leq  \esp{(H_0^\eps)^\gamma  }
    + \PAR{\osc(V_\eps) +\frac{\gamma}{\beta}} 
      \alpha \esp{  \int_0^{t\wedge\tau_n}  (H_s^\eps)^{\gamma -1} \me^{-\alpha(t-s)} ds} \\
    &\leq  \esp{(H_0^\eps)^\gamma  }
    + \PAR{\osc(V_\eps) + \frac{\gamma}{\beta} } \sup_{s \leq t }\esp{(H_s^\eps)^{\gamma -1}}  ds.   
 \end{align*}
 Sending $n$ to infinity, we apply Fatou's lemma
 to get 
 \[
   \esp{ (H_t ^\eps)^\gamma } 
   \leq \esp{(H_0^\eps)^\gamma}
    + \PAR{\osc(V_\eps)  + \frac{\gamma}{\beta}}
    \sup_{s \leq t }\esp{(H_s^\eps)^{\gamma -1}},
  \]
    and thus
    \begin{equation}
      \label{eq:recc_gamma}
      \sup_{t \geq 0}\esp{ (H_t ^\eps)^\gamma  }
      \leq \esp{(H_0^\eps)^\gamma   }
    + \PAR{\osc(V_\eps)  
      + \frac{\gamma }{\beta}} \sup_{ t \geq 0 }\esp{(H_s^\eps)^{\gamma -1}}.
  \end{equation}
  We are now ready to conclude. Say that $\gamma$ is good if 
  there exists a $C(\gamma,M,\beta)$ such that for all $\eps$, 
  \[\sup_t\esp{(H^\eps_t)^\gamma} \leq C(\gamma,M,\beta) (1 + \esp{(H^\eps_0)^\gamma}, \]
  whenever $\osc(V_\eps) \leq M$.  The bound \eqref{eq:recc_gamma}
  immediately shows that $\gamma=1$ is good. 
  If $\gamma$ is good and $\gamma \leq \gamma'\leq \gamma + 1$, using the elementary
  inequality $x^a\leq 1+x^b$ valid for any $x>0$ and any $1\leq a<b$, we get
    \begin{align*}
      \sup_{t \geq 0}\esp{ (H_t ^\eps)^{\gamma'}  }
      & \leq \esp{(H_0^\eps)^{\gamma'}   }
    + \PAR{\osc(V_\eps)  
      + \frac{\gamma' }{\beta}} \sup_{ t \geq 0 }\esp{(H_s^\eps)^{\gamma' -1}} \\
      & \leq \esp{(H_0^\eps)^{\gamma'}   }
    + \PAR{ M   + \frac{\gamma' }{\beta}} \PAR{  1 +  \sup_{ t \geq 0 }\esp{(H_s^\eps)^{\gamma}}} \\
      & \leq \esp{(H_0^\eps)^{\gamma'}   }
    + \PAR{ M    + \frac{\gamma' }{\beta}}( 1 +  C(\gamma,M,\beta) \esp{(H^\eps_0)^\gamma}) \\
      & \leq \esp{(H_0^\eps)^{\gamma'}   }
    + \PAR{ M    + \frac{\gamma' }{\beta}}\PAR{  1 +  C(\gamma,M,\beta) \PAR{ 1 + \esp{(H^\eps_0)^{\gamma'}}}}
  \end{align*}
  showing that $\gamma'$ is itself good. Therefore all $\gamma\geq 1$ are good. 
  Using the bounds $(1/2)p^2 \leq H^\eps(q,p) \leq (1/2) p^2 + M$ it is easy to
  translate this into bounds on $\esp{\abs{P^\eps_t}^{2\gamma}}$, concluding 
  the proof of Lemma~\ref{lem:finiteMoments}. 
\end{proof}

\begin{proof}[Proof of Lemma~\ref{lem:boundsOnPt}]
  Let us fix an arbitrary $T > 0$, and prove~\eqref{eq:espSupPt},
  that is, prove the existence of a numerical constant $C(\beta,M,T)$ such 
  for any $\eps \in (0,1)$, 
  \begin{align}
   \label{eq:espSupPt_bis}
   &\esp{ \sup_{0\leq t\leq T}|P^\eps_t|^2}  
     \leq \esp{|P^\eps_0|^2} + \frac{1}{\eps} C(\beta, M ,T) \PAR{\esp{|P^\eps_0|^2}+1}^{1/2}
\end{align}
whenever $\osc(V_\eps) \leq M$. As before, since 
$2 H^\eps_t - 2  M  \leq (P^\eps_t)^2 \leq 2 H^\eps_t$, 
it is enough to prove the statement with $H^\eps_t$ instead of $\abs{P^\eps_t}^2$. 

We start by recalling~\eqref{eq:boundHk} for $\gamma=1$ and $\alpha = 2/\eps^2$: 
\begin{align}
  \label{eq:decompositionHtbis}
  H^\eps_t &\leq
  H^\eps_0  +  \PAR{  \osc(V_\eps)  + \frac{1}{\beta}} 
      + \frac{\sqrt{2\beta^{-1}}}{\eps} \int_0^t e^{-\alpha(t-s)}  P^\eps_s dW_s. 
\end{align}
Recall that this led by a localization argument to the 
following bound \eqref{eq:recc_gamma}: 
    \begin{equation}
      \label{eq:boundEspHt}
      \sup_{t \geq 0} \esp{ H_t ^\eps  }
      \leq \esp{H_0^\eps}
    + \PAR{ M  + \frac{1}{\beta}}. 
  \end{equation}
In order to control the expectation of the supremum, we must control the stochastic integral. 
Define $M_t = \int_0^t P^\eps_s dW_s$ and integrate by parts: 
\begin{align*}
  \abs{ \int_0^t e^{-\alpha(t-s)} P^\eps_s dW_{s}}
    &=\abs{  \int_0^t e^{-\alpha(t-s)}  dM_s }
    = \abs{ M_t - \alpha \int_0^t e^{-\alpha (t-s)} M_s ds} \\
    &\leq \abs{M_t} + \sup_{s\in[0,t]} \abs{M_s} \\
    &\leq 2 \sup_{s\in [0,T]} \abs{M_s}.
\end{align*}
Plugging this in \eqref{eq:decompositionHtbis} yields
  \begin{equation} \label{eq:decompositionHtter}
    \sup_{t\in [0,T]} H_t 
    \leq 
  H^\eps_0  +  \PAR{  \osc(V_\eps) + \frac{1}{\beta}} 
      + \frac{\sqrt{2\beta^{-1}}}{\eps} 2\sup_{t\in [0,T] }\abs{M_t}. 
    \end{equation}
By Doob's martingale maximal inequality, Itô's isometry and the bound~\eqref{eq:boundEspHt} we get
\begin{align*}
  \esp{ \sup_{0\leq t\leq T} |M^\eps_t|^2}
  &\leq 4\esp{ |M^\eps_T|^2}
   = 4\esp{\abs{ \int_0^T P^\eps_s dW_s }^2} = 4\esp{ \int_0^T (P^\eps_s)^2 ds }\\
  &\leq 8T\PAR{ \osc(V_\eps)  + \sup_{t\in [0,T]} \esp{H^\eps_t}} \\
  &\leq 8T\PAR{ 2 \osc(V_\eps) + \esp{H^\eps_0} + \frac{1}{\beta}}. 
\end{align*}
Injecting this in~\eqref{eq:decompositionHtter} and applying Cauchy--Schwarz inequality 
yields
\begin{equation}
    \esp{\sup_{t\in [0,T]} H_t }
    \leq 
  \esp{ H^\eps_0 } +  \PAR{  \osc(V_\eps )  + \frac{1}{\beta}} 
      + 8 \frac{\sqrt{T\beta^{-1}}}{\eps} \PAR{ 2 \osc(V_\eps)  + \esp{H^\eps_0} + \frac{1}{\beta} }^{1/2}, 
\end{equation}
concluding the proof of~\eqref{eq:espSupPt_bis}. 
\end{proof}

\appendix

\section{Stopped martingale problem}
Let $E$ be a Polish space. Let $L$ be a linear operator mapping a given space
$\cD \subset C_b(E)$ into bounded measurable functions. Let $\mu$ be a
probability distribution on $E$. Let $U \subset E$ be an open set. A
c\`ad-l\`ag process $(X_t)_{t\geq 0}$ with values in  $E$ solves the
\emph{stopped martingale problem} for the generator $L$ on the space  $\cD$
with initial measure $\mu$ and domain $U$ ---~in short, $X$ solves
$\sMP(L,D(L),\mu,U)$~--- if, denoting
\[
  \tau_U \equalsDef \inf\set{t \geq 0 | X_t \notin U \text{\, or \,} X_{t^-} \notin U}, 
\]
(i) $\law(X_0) = \mu$; (ii) $X_{t}=X_{t \wedge \tau_U}$; and (iii) if for any $\varphi\in \cD$, 
\begin{equation*}
t\mapsto M_t(\varphi) \equalsDef \varphi(X_t)-\varphi(X_0)-\int^{t \wedge \tau_U}_0 L\varphi(X_s)ds
\end{equation*}
is a martingale with respect to the natural filtration $\PAR{\cF^X_t=\sigma\PAR{X_s, \, 0 \leq s \leq t}}_{t \geq 0}$. 

Moreover, the stopped martingale problem $\sMP(L,\cD,\mu,U)$ is said to be \emph{well-posed} if:
\begin{itemize}
  \item There exists a probability space and a  c\`ad-l\`ag process defined on it 
    that solves the stopped martingale problem (existence); 
  \item whenever two processes solve $\sMP(L,\cD,\mu,U)$, then they have the 
    same distribution on $\mathbb{D}_E$ (uniqueness).
\end{itemize}

The following theorem is a synthesis of the localization technique of Theorem~$6.1$ and~$6.2$ of~\cite[Chapter~$4$]{EK}. It gives a simple criteria ensuring equivalence of uniqueness between (i) a global martingale problem, and (ii) local stopped martingale problems.

\begin{theorem}\label{th:local_mart}
 Let $(U_k)_{k \in K}$ be a countable family of open subsets of $E$ such that
 $\bigcup_{k \in K} U_k =E$. Assume that for any initial $\nu$, there exists a
 solution to $\MP(L,\cD,\mu)$. Then uniqueness of $\MP(L,\cD,\mu)$ for all
 $\mu$ is equivalent to uniqueness of $\sMP(L,\cD,\mu,U_k)$ for all $\mu$ and
 all $k \in K$.
\end{theorem}

\bibliographystyle{amsalpha}
\bibliography{weaklimit}

\end{document}